\documentclass[a4paper,12pt,oneside,reqno]{amsart}
\usepackage{hyperref}
\usepackage[headinclude,DIV13]{typearea}
\areaset{15.1cm}{25.0cm}
\usepackage{txfonts,amssymb,amsmath,amsthm,bbm}
\usepackage[latin1] {inputenc}
\usepackage{graphicx, psfrag}
\usepackage{xcolor}
\usepackage{subfigure}

\newtheorem{theorem}{Theorem}[section]
\newtheorem{lemma}[theorem]{Lemma}

\newtheorem{definition}[theorem]{Definition}

\definecolor{bbm}{RGB}{51,153,0}
\definecolor{above}{RGB}{128,0,128}
\definecolor{below}{RGB}{102,0,204}
\definecolor{cascade}{RGB}{204,0,0}
\definecolor{iid}{RGB}{153,51,0}

\theoremstyle{remark}
\newtheorem*{remark}{Remark}

\def\paragraph#1{\noindent \textbf{#1}}

\numberwithin{equation}{section}

\def\d{\mathrm{d}}

\def\<{\langle}
\def\>{\rangle}

\def\e{\epsilon}

\def\g{\gamma}

\def\l{\lambda}

\def\s{\sigma}
\def\t{\tau}

\def\D{\Delta}

\def\del{\partial}
\def\R{{\Bbb R}}  
\def\N{{\Bbb N}}  
\def\P{{\Bbb P}}  
\def\Z{{\Bbb Z}}

\def\E{{\Bbb E}}

\let\cal=\mathcal
\def\AA{{\cal A}}
\def\BB{{\cal B}}

\def\EE{{\cal E}}
\def\FF{{\cal F}}
\def\GG{{\cal G}}

\def\TT{{\cal T}}

\def\UU{{\cal U}}

 \def \s {{\sigma}}

 \def \D {{\Delta}}
 \def \t {{\tau}}

 \def \g {{\gamma}}
 \def \l {{\lambda}}
 \def \d {{\delta}}

 \def \del {{\partial}}

 \def \ba {\begin{array}}
 \def \ea {\end{array}}

 \newcommand{\be}{\begin{equation}}
 \newcommand{\ee}{\end{equation}}

\newcommand{\bea}{\begin{eqnarray}}
 \newcommand{\eea}{\end{eqnarray}}
\def\TH(#1){\label{#1}}\def\thv(#1){\ref{#1}}
\def\Eq(#1){\label{#1}}\def\eqv(#1){(\ref{#1})}

 \def \1{\mathbbm{1}}

\def\wh{\widehat}

\def\eee{{\mathrm e}}

 \DeclareMathOperator{\sech}{sech}

\begin{document}


 \title[BBM with self-repulsion]{Branching Brownian motion with self repulsion}
\author[A. Bovier]{Anton Bovier}
 \address{A. Bovier\\Institut f\"ur Angewandte Mathematik\\
Rheinische Friedrich-Wilhelms-Universität\\ Endenicher Allee 60\\ 53115 Bonn, Germany }
\email{bovier@uni-bonn.de}
\author[L. Hartung]{Lisa Hartung}
 \address{L. Hartung\\
  Institut für Mathematik \\Johannes Gutenberg-Universität Mainz\\
Staudingerweg 9,
55099 Mainz, Germany}
\email{lhartung@uni-mainz.de}

\date{\today}

 \begin{abstract}  
 We consider a model of branching Brownian motion with self repulsion. Self-repulsion is introduced via change of measure that penalises particles spending time in an $\e$-neighbourhood of   each other. We derive a simplified version of the model 
 where only branching events are penalised. This model is almost exactly solvable and we derive a precise description 
 of the particle numbers and branching times.   In the limit of weak penalty, an interesting universal time-inhomogeneous 
 branching process emerges. The position of the maximum is governed by a F-KPP type reaction-diffusion equation with 
 a time dependent reaction term. 
 
 \end{abstract}

\thanks{
This work was partly funded by the Deutsche Forschungsgemeinschaft (DFG, German Research Foundation) under Germany's Excellence Strategy - GZ 2047/1, Projekt-ID 390685813 and GZ 2151 - Project-ID 390873048,
through Project-ID 211504053 - SFB 1060, through Project-ID 233630050 -TRR 146, through Project-ID 443891315  within SPP 2265, and Project-ID 446173099.
 }


 \maketitle


\section{Introduction}
 Branching Brownian motion (BBM)  \cite{Moyal62,AdkeMoyal63} can be seen as an elementary model
 for the evolution of a population of individuals that are subject to birth, death, and motion in space.
 One of the primary interests in this model was the analysis of the speed of spread of
 such a population in space, as well as finer properties of the front. 
 Indeed,   BBM was investigated form the point of view of extreme value theory 
over the last 40 year, see,  e.g.,
\cite{B_M,LS,chauvin88,chauvin90,ABK_G,ABK_P,ABK_E,ABBS,CHL17,bbm-book}.

As a model for population dynamics, BBM is somewhat unrealistic as it leads to uncontrolled 
exponential growth of the population size. In fact, in the standard normalisation, the 
population size grows like $\exp(t)$, while the population spreads over a volume of order $t$, 
leading to an unsustainable density of the population. Several variants of the model that resolve this 
problem have been proposed where, according to some selection rule,  offspring is selected to 
survive in such a way that the total population size stays controlled \cite{BD97,Mallein17,CorMal17,maillard2020}. 
Versions where competitive interactions between particles are present were considered, e.g.  in 
\cite{Engl2010,Engl2015,Penington2017,Penington2019}.

In this paper, we propose a model where the population size is controlled by penalising 
the fact that particles stay close to each other. Before defining the model precisely,
 recall that BBM is constructed as follows:
 start with a single particle which performs a standard Brownian motion $x(t)$  in $\R$ 
 with $x(0)=0$ and continues for a standard exponentially distributed  holding time $T$,
independent of $x$. At time $T$, the particle splits independently of $x$ and $T$ into $k$ offspring with probability $p_k$, where $\sum_{i=1}^\infty p_k=1$,
$\sum_{k=1}^\infty k p_k=2$ and $K=\sum_{k=1}^\infty k(k-1)p_k<\infty$.  In the present paper,
we choose the simples option, $k_2=1$, all others zero, except in Section 8, where
we allow for $p_0>0$. These particles continue along independent Brownian paths starting from $x(T)$ and are subject to the same splitting rule.
And so on. We let $n(t)$ denote the number of particles at time $t$ and label 
the particles at time $t$    arbitrarily by  $1,2,3,\dots, n(t)$, and denote by  $x_1(t),\dots, x_{n(t)}(t)$ 
the positions of these particles at that time. For $s\leq t$, we let $x_i(s)$ be the position of the 
ancestor of particle $i$ at time $s$. We  denote by $\P$ the law of BBM.

Alternatively, BBM can be constructed as a Gaussian process indexed by a  
continuous time 
Galton-Watson tree with mean zero and covariances, conditioned on the Galton-Watson tree, given by
\be\Eq(variance.2.1)
\E\left[ x_k(s)x_\ell(r)|\s(GW)\right]= d(x_k(s),x_\ell(t))\wedge s\wedge r,
\ee
where $d(x_k(t),x_\ell(t))$ is the 
time of the most recent common ancestor of the particles labeled $k$ and $\ell$ in the
Galton-Watson tree.

For $t<\infty$ and for some $\e>0$, we define the penalty function 
\be\Eq(int.1)
I_t(x) \equiv \int_0^t\sum_{i\neq j=1}^{n(s)} \1_{|x_i(s)-x_j(s)|\leq \e}ds.
\ee
(The notation here is not quite consistent, as the labelling of the $n(s)$ particles at time $s$ 
is changing with $s$. This can be remedied by using the Ulam-Kesten-Harris labelling of
the tree, but maybe this is not necessary here.)
We are interested in the law of $x_t$ under the tilted measure $P_{t,\s}$ given by
\be\Eq(tilt.1)
P_{t,\s} (A) \equiv \frac {\E\left[\1_{x_t\in A} \eee^{-\l I_t(x)}\right]}
{\E\left[\eee^{-\l I_t(x)}\right]},
\ee
for any  Borel set  $A$.
The function $I_t$ measures the total time when any two particles stay within a distance $\e$ up to time $t$. This seems to a be reasonable measure for competitive pressure. 
In a typical realisation of BBM, the density of particles at time $s$ will be of order $\eee^s/s$, and hence the $\epsilon$-neighbourhood of any particle contains $\e\eee^s/s$ other particles. Thus, for a typical configuration $x$ of BBM,  $I_t(x)\sim \e\eee^{2t}/(2t)$. This penalty is most easily avoided by reducing the particle number by not branching. For a particle to not branch up to time $t$ has probability $\eee^{-t}$, which is far less costly. Reducing the particle density by making the
particles move much farther apart would be far more costly.

\paragraph{A simplified model.}
Analysing the measure $P_{t,\s}$ directly seems rather difficult. We suggest an approximation that 
should share the qualitative features of the full measure. For this we consider a lower bound on 
$I_t$. Note that, whenever branching occurs, the offspring start at the same point and thus are all 
closer than $\e$ Let us for simplicity take a branching law such that $p_2=1$, i.e. 
only binary branching occurs. Then we can bound
\be\Eq(approx.1)
I_t(x) \geq  I'_t(x)\equiv\sum_{i=1}^{n(t)-1} \t_{\e}(i),
\ee
where $\t_\e(i)$ is the first time the two Brownian motions that start at the $i$-th branching event
are a distance $\e$ apart.  For small $\e$, the probability that one of the two branches branches again before the time $\t_\e$
is of order $\e^2$, so that it will be a good approximation to treat the 
 $ \t_{\e}(i)$ as independent and having the same distribution as 
\be
\Eq(tau.1)
\t_\e\equiv \inf\{t>0: |B_t|>\e\}.
\ee
Then, 
\be\Eq(approx.2)
\E \left[\eee^{-\l  \sum_{i=1}^{n(t)-1} \t_{\e}(i)}\right] \approx
\E\left[\E \left[\eee^{-\l \t_{\e}}\right]^{n(t)-1}\right]\equiv \E\left[ \s^{n(t)-1}\right].
\ee
 But (as follows from Theorem 5.35 and Proposition 7.48 in \cite{MoePer2010}), 
 \be
 \Eq(approx.3)
 \E \left[\eee^{-\l \t_{\e}}\right] =\sech(\e \sqrt {2\l}),
 \ee
 which for small $\l\e^2$ behaves like $\exp(-\l\e^2)$. Note that we also have, by Jensen's inequality, that 
 \be\Eq(approx.3.1)
 \E \left[\eee^{-\l  \sum_{i=1}^{n(t)-1} \t_{\e}(i)}\right]\leq  \E \left[\eee^{-\l  \sum_{i=1}^{n(t)-1}\E [\t_{\e}(i)]}\right]=
 \E \left[\eee^{-\l \e^2(n(t)-1)}\right].
 \ee
One might think that the approximate model is a poor substitute for the full model, since it ignores 
the repulsion of particles after the time that they first separate. However, as we will see shortly, 
already $I'_t(x)$ suppresses branching so much that the total number of particles will stay finite for any time. Hence we can expect that these finitely many particles can remain separate rather easily 
and that the remaining effect of $I_t$ will be relatively mild. 

\paragraph{Outline.} The remainder of this paper is organised as follows. In Section 2 we derive exact formulas for the partition function, the particle number, and the first branching time in the simplified model. In Section 3 we introduce the notion
of \emph{quasi-Markovian Galton-Watson trees}. In Section  4   we show that the branching times in the  simplified model 
are given by such a tree. In Section 5, we consider the limit when $\l\downarrow 0$ and derive a universal asymptotic
model, which is a specific  quasi-Markovian Galton-Watson tree. In Section 6, we consider the position of the maximal 
particle and show that its distribution is governed by a F-KPP equation with time dependent reaction term 
and analyse the behaviour of its solutions.
We discuss the relation of the approximate model to the full model in Section 7. 
In Section 8, we briefly look at the case when $p_0>0$. In this case, the process dies out and we derive the rate at which the
number of particles tends to zero.

\section{Partition function, particle numbers, and first branching time }

\subsection{The partition function}
The first object we consider is the normalising factor or \emph{partition function}
\be\Eq(pf.1)
v_\s(t)\equiv 
\E \left[ \s(\l,\e)^{n(t)-1}\right].
\ee
\begin{lemma}\TH(pf.2) Let
$\tilde v_\s(t)$ be the solution of the ordinary differential equation 
\be\Eq(pf.3)
\frac d{dt} \tilde v_\s(t) =  \s(\l,\e)  \tilde v_\s(t)^2 -\tilde v_\s(t),
\ee
with initial condition $\tilde v_\s(0)=1$.
Then 
$ v_\s(t)=\tilde v_{\s(\l,\e)}(t)$.
\end{lemma}

\begin{proof}
The derivation of the ode is similar to that of the F-KPP equation for BBM  (see \cite{bbm-book}). Clearly, $v_\s(0)=1$. 
Conditioning on the time if the first branching event, we get 
\be\Eq(pf.4)
v_\s(t) =\eee^{-t} + \int_0^t ds \eee^{-(t-s)} \s(\l,\e) v_\s(s)^2.
\ee
Differentiating with respect to $t$ gives 
\bea\Eq(pf.5)
\frac d{dt} v_\s(t) &=&-\eee^{-t} +\s(\l,\e)  v_t(t)^2 -\int_0^t ds \eee^{-(t-s)} \s(\l,\e) v_\s(s)^2
\nonumber\\
&=&\s(\l,\e) v_\s(t)^2 -v_\s(t).
\eea
Thus $v_\s$ solves the equation \eqv(pf.4) with $s=\s(\l,\e)$, which proves the lemma.
\end{proof}

A first inspection of Eq. \eqv(pf.3) shows why the cases $\l=0$ and $\l>0$ are vastly different.

Equation \eqv(pf.3) has the two fix points $0$ and $1/\s$. Here $0$ is stable and $ 1/\s$ is unstable.
Hence all solutions with initial condition $0\leq v_\s(0)<1/\s$ will converge to $0$, while solutions with $ v_\s(0)>1/\s$ will tend to infinity. Only the special initial condition $v_\s(0)=1/\s$ will lead to
the constant solution. Since we start with the initial condition $v_\s=1$, 
if $\l=0$ and hence $\s=1$, we get this special constant solution, while for $\l>0$, the solution will tend to zero. 
In fact, we can solve \eqv(pf.3) exactly. To do this it is convenient to define 
\be\Eq(pf.6)
\hat v_\s(t)\equiv \eee^t \tilde v_\s(t).
\ee
Then $\hat v_\s$ solves 
\be\Eq(pf.7)
\frac d{dt}\hat v_\s(t) =  \s \hat v_\s(t)^2\eee^{-t},
\ee 
also with initial condition $\hat v_\s(0)=1$. Dividing both sides by $\hat v_\s^2$, 
this can be written as
\be
\Eq(pf.8)
- \frac  d{dt} \frac 1{\hat v_\s(t)} =\s \eee^{-t},
\ee
which can be integrated to give
\be\Eq(pf.9)
-\frac 1{\hat v_\s(t)}+\frac 1{\hat v_\s(0)} = \s\left(1- \eee^{-t}\right),
\ee
or 
\be\Eq(pf.9.1)
{\hat v_\s(t)} =\frac 1{ \frac 1{\hat v_\s(0)} -\s\left(1- \eee^{-t}\right)},
\ee
and 
\be\Eq(pf.9.2)
{\tilde v_\s(t)} =\frac 1{ \eee^{t}\left( \frac 1{\tilde v_\s(0)} -\s\right)+\s}.
\ee

Using the initial condition $\hat v_\s(0)=1$,  we get 
\be\Eq(pf.10)
\tilde v_\s(t)=  \frac{\eee^{-t}}{ \left(1 - \s\right)+\s\eee^{-t}}.
\ee
Thus, provided $\s<1$, 
\be\Eq(pf.11)
\lim_{t\uparrow\infty}
\hat v_\s(t)=  \frac{1}{1 -\s }.
\ee

\subsection{Particle numbers}
From the formula for the partition function we can readily infer the mean number of particles at time $t$, nramely,
\be\Eq(mean.1)
\wh E_{\l,t} n(t)=1+ \s\frac d{d\s} \ln \hat v_\s(t) =
1+ \s\frac{1-\eee^{-t}}{1-\s(1-\eee^{-t})}=\frac 1{1-\s(1-\eee^{-t})},
\ee
and for $t\uparrow \infty$ this converges to $1/(1-\s)$. For small $\l\e^2$, this in turn behaves like 
$(\l\e^{2})^{-1}$.

\def\ttt{q}
In fact, we can even compute the distribution of the number of particles at times $s\leq t$. 
To do so, we want to compute the Laplace (Fourier) transforms
\be\Eq(s.1)
\wh E_{t,\s}\left[\eee^{\g n(s)}\right]=\frac {\E \left[\eee^{\g n(s)}\s^{n(t)-1}\right]}{\E \left[\s^{n(t)-1}\right]} =\frac {\E \left[\eee^{\g n(s)}\s^{n(t)-1}\right]}{v_\s(t)}.
\ee
The denominator has already been calculated. For the numerator we write
\bea\Eq(s.2)
\E \left[\eee^{\g n(s)}\s^{n(t)-1}\right]&=&
\E \left[\eee^{\g n(s)}   \E\left[\s^{n(t)-1}\big|\FF_s\right]\right]\nonumber\\
&=&\E \left[\eee^{\g n(s)}   \E\left[\s^{ \sum_{i=1}^{n(s)} n^{(i)} (t-s)-1}\big |\FF_s\right]\right],
\eea
where $n^{(i)} (t-s)$ are the number of particles at time $t$ that have particle $i$ as common ancerstor at time $s$. 
Using the  independence properties,
this equals
\bea\Eq(s.3)
&&\E \left[\eee^{\g n(s)} \s^{n(s)-1} \left( \E\left[\s^{ n (t-s)-1}\right]\right)^{n(s)}\right]\nonumber\\
&&= \eee^{\g}v_\s(t-s)\E \left[\left(\eee^{\g}\s v_\s(t-s)\right)^{n(s)-1}\right]\nonumber\\
&&=  \eee^{\g}v_\s(t-s) \frac {\eee^{-s}}{1-\eee^{\g}\s v_\s(t-s)(1-\eee^{-s})}\nonumber\\
&&= \frac {\eee^{-s}}{\eee^{-\g}v_\s(t-s)^{-1}-\s(1-\eee^{-s})}\nonumber\\
&&= \frac {\eee^{-s}}{\eee^{-\g}\left(1-\s(1-\eee^{-t+s})\right)\eee^{t-s}-\s(1-\eee^{-s})}\nonumber\\
&&= \frac {\eee^{-t}}{\eee^{-\g}\left(1-\s(1-\eee^{-t+s})\right)-\s(\eee^{s-t}-\eee^{-t})}.
\eea
Dividing by $v_\s(t)$, we arrive at 
\be\Eq(s.4)
\wh E_{t,\s}\left[\eee^{\g n(s)}\right]=
 \frac {1-\s(1-\eee^{-t})}{\eee^{-\g}\left(1-\s(1-\eee^{-t+s})\right)-\s(\eee^{s-t}-\eee^{-t})}.
 \ee
 From this  exact formula we can derive various special cases.

\begin{theorem}\TH(geo.1)
\begin{itemize}
\item[(i)]
Under the measure $\wh P_{\l,t}$, the number of particles at time $t$ is geometrically distributed with parameter $1-\s(\l,\e)(1-\eee^{-t})$.
In particular, the number of particles converges, as $t\uparrow\infty$, to a geometric random variable with parameter $1-\s(\l,\e)$. 
\item[(ii)] As $t\uparrow \infty$, the number of particles at time $s(t) = t+\ln(1-\s)+\rho$, for all 
$\rho\leq -\ln(1-\s)$,
converges in distribution to a geometric random variable with parameter\\ $(1+\s\eee^{\rho})^{-1}$.

\end{itemize}
\end{theorem}
\begin{proof}
 Inserting $s=t$ into \eqv(s.4) we get that 
 \be\Eq(s.t)
\wh E_{t,\s}\left[\eee^{\g n(t)}\right]=
 \frac {1-\s(1-\eee^{-t})}{\eee^{-\g}-\s(1-\eee^{-t})},
 \ee
 which is the Laplace transform of the geometric distribution with parameter
  $1-\s(\l,\e)(1-\eee^{-t})$. This implies (i). 
  Similarly,   with $s=t+\ln (1-\s) +\rho$, and $\rho\leq -\ln(1-\s)$
 \be\Eq(s.5)
\wh E_{t,\s}\left[\eee^{\g n(s)}\right]=
 \frac {1-\s(1-\eee^{-t})}{\eee^{-\g}\left(1-\s+\s(1-\s)\eee^{\rho}\right)-\s(1-\s)(\eee^{\rho}-\eee^{-t})}.
 \ee
  If we now take $t\uparrow\infty$, we get 
\be\Eq(s.6)
 \lim_{t\uparrow\infty}
 \wh E_{t,\s}\left[\eee^{\g n(t+\ln(1-\s)+\rho)}\right]=
  \frac {1}{\eee^{-\g}\left(1+\s\eee^{\rho}\right)-\s\eee^{\rho}}
  =  \frac {(1+\s\eee^{\rho})^{-1}}{\eee^{-\g} -\frac{\s\eee^{\rho}}{1+\s\eee^{\rho}}},
  \ee
which is the Laplace transform of the geometric distribution with parameter $(1+\s\eee^{\rho})^{-1}$.
  
 \end{proof}
 
 \begin{remark}
 Note that, for fixed $s$, taking the limit $t\uparrow \infty$, we get unsurprisingly 
 $\eee^{\g}$, indicating that there is just one particle.

We see that the mean number of particles ranges from $1$ (as $\rho\downarrow -\infty$), $1+\s$ (for $\rho=0$), 
to $1-\s$ (for $\rho=-\ln (1-\s)$).  
Note that if $\s=1$, $n(t)$ is geometric with parameter $\eee^{-t}$, which corresponds to 
 BBM with binary branching.

\end{remark}

\subsection{Distribution of the first branching time}
We have seen so far that the repulsion strongly suppresses the number of branchings.
 The first branching time is then 
\be 
\t_1\equiv \inf\{ s>0:n(s)=2\}.
\ee

\begin{theorem}\TH(branch.1)
The distribution of the first branching time under $\wh P_{\l,t}$ is given by
 \be\Eq(branch.2)
 \wh P_{\l,t} \left(\t_1\leq t-r\right)  =\s(\l,\e)
 \frac {\eee^{-r}-\eee^{-t}}{1-\s(\l,\e)(1-\eee^{-r})}.
 \ee
\end{theorem}

\begin{proof}
 Note that after the first branching, there will be two independent BBMs that run for the remaining 
 time $t-\t_1$ and that are subject to the same penalty as before. In particular, given $\t_1$, 
 the total particle number $n(t)$ is equal to the sum of the number of 
 particles in these two branches, 
 \be
 \Eq(rec.1)
 n(t)= \tilde n^{(0)}(t-\t_1)+ \tilde n^{(1)}(t-\t_1),
 \ee
 where $ \tilde n^{(i)}$ are the particles in the two branches that split at time $\t_1=e_0(0)$. 
Denote by $v_\s(t,r)$ the unnormalised mass of paths that branch before time $t-r\leq t$, i.e. 
set
\be\Eq(gen.1)
v_\s(t,t-r)\equiv \E \left[\s(\l,\e)^{n(t)-1}\1_{\t_1\leq t- r}\right].
\ee
We get 
\bea\Eq(gen.2)
v_\s(t,t-r)&=& \int_0^{t-r} \eee^{-s} \E \left[ \s(\l,\e)^{\tilde n^{(0)}(t-s)+ \tilde n^{(1)}(t-s) -1}\right] ds\nonumber\\
&=& \int_0^{t-r} \eee^{-s} \s(\l,\e)
 \E \left[ \s(\l,\e)^{ n(t-s)-1}\right]\E\left[\s(\l,\e)^{ n(t-s)-1}\right]ds\nonumber\\ 
 &=&
\int_0^{t-r} ds \eee^{-s} \s(\l,\e) v_\s(t-s)^2ds. 
 \eea
 Since $v_\s(t-s)$ is known, this is an explicit formula, 
 namely 
\bea\Eq(gen.3)
 v_\s(t,t-r) &=&\s(\l,\e) \eee^{-2t} \int_0^{t-r} ds
 \eee^{s}  \frac 1{(1-\s(\l,\e)(1-\eee^{-(t-s)}))^2}\nonumber\\
 &=&\eee^{-t}\s(\l,\e)\frac {\eee^{-r}-\eee^{-t}}{(1-\s(\l,\e)(1-\eee^{-t}))(1-\s(\l,\e)(1-\eee^{-r}))}.
 \eea
 Since 
  $
 \wh P_{\l,t} \left(\t_1\leq t-r\right) =\frac {v_\s(t,t-r)}{v_\s(t)}$, 
  \eqv(branch.2) follows. 
  \end{proof}

%
%
%
\begin{remark} Note that, for $r$ fixed, $\wh P_{\l,t} \left(\t_1\leq t-r\right)$ converges, as $t\uparrow\infty$,
 to 
\be\Eq(gen.5)
 \frac {\s(\l,\e)\eee^{-r}}{1-\s(\l,\e)(1-\eee^{-r})}.
\ee
 Note further  that  $v_\s(t)= v_\s(t,t) +\E\1_{\t_1>t}=  v_\s(t,t) +\eee^{-t}$ and therefore
  \be\Eq(gen.3.2)
    \wh P_{\l,t} \left(\t_1\leq t\right) =\frac {v_\s(t,t)}{v_\s(t)}
    =1-\eee^{-t}/v_\s(t)<1.
    \ee
    \end{remark}
\def\ttt{q}
\section{Quasi-Markovian time-inhomogeneous Galton-Watson trees}

In this section we introduce a class of models that are continuous-time version of Galton-Watson processes
that are time-inhomogeneous and that in general are not Markov, but have an underlying discrete-time Markov 
property. These processes emerge in the models introduced above. 

We  start with discrete time trees and we introduce the usual Ulam-Harris labelling. 

 Let us define the set of (infinite) multi-indices
\be\Eq(multi.1)
\mathbf{I}\equiv \Z_+^\N,
\ee
and let $\mathbf{F}\subset \mathbf{I}$ denote the subset of multi-indices that contain 
only  finitely many entries that are different from zero. Ignoring leading zeros, we see
that 
\be\Eq(multi.2)
\mathbf{F} = \cup_{k=0}^\infty \Z_+^k,
\ee
where $\Z_+^0$ is either the empty multi-index or the multi-index containing only 
zeros.
A discrete-time tree is then identified by a consistent sequence of sets of multi-indices, $\ttt(n)$ at time $n$ as follows.

\begin{itemize}
 \item $\{(0,0,\dots)\} =\{u(0)\}=\ttt(0)$.
 \item If $u\in \ttt(n)$ then $u+(\underbrace{0,\dots,0}_{n\times 0}, k,0,\dots)\in 
 \ttt(n+1)$ if $0\leq k\leq l^u(n)-1$, where
 \be\Eq(map.1)
l^u(n)=\# \{\mbox{ offsprings of  the particle corresponding to }u\, \mbox{at time}\, n \}.
\ee
\end{itemize}
We can relate the assignment of labels in a backwards consistent fashion as follows.
For  $u\equiv (u_1,u_2,u_3,\dots)\in \Z_+^\N$, we define the function $u(r), r\in \R_+$,  through
\be\Eq(multi.3)
u_\ell(r)\equiv \begin{cases}  u_\ell,&\,\, \hbox{\rm if}\,\, \ell\leq r,\\
0,&\,\, \hbox{\rm if}\,\, \ell> r.
\end{cases}
\ee
Clearly,  if $u(n)\in \ttt(n)$ and  $r\leq t$, then $u(r)\in \ttt(r)$.  This allows to define the 
\emph{boundary} of the 
tree at infinity  as follows:
\be\Eq(multi.4)
\del \mathbf{T} \equiv \left\{ u\in \mathbf{I}: \forall n<\infty, u(n)\in \ttt(n)\right\}.
\ee

We also want to be able to consider a branch of a tree as an entire new tree. For this we use the notation 
$\overleftarrow{u} =(u_1,u_2,u_3,\dots)$ if $u=(u_0,u_1,u_2,\dots)$.

Given a discrete-time tree, we can turn it into a continuous-time tree by assigning 
waiting times to each vertex, resp. to each multi-index in the tree. E.g., in the case of the standard continuous-time 
Galton-Watson tree, we simply assign standard, iid, exponential random variables, $e_u(n)$,  to each vertex, resp. multi-index. Note that we choose the notation in such a way that we think of $u$ as an element of the boundary of the tree, and $e_u(n)$ is the waiting time attached to the vertex labelled $u(n)$ (in the $n$-th generation).
This time represents the waiting time from the birth of this branch to its next branching.
This assigns a total time, $T_u(n)$  for the branching of a multi-index at discrete time $n$, as
\be\Eq(tree.1)
T_u(n)=t_0+\sum_{k=0}^n e_u(k),
\ee
where $t_0\in\R$ is an initial time associated to the root of the tree and
  $e_{u}(0) $ is the time of the first branching of the root of the tree.
We denote by $\FF_n$ the $\s$-algebra generated by the branching times of the first $n$ generations of the tree, i.e.
\be\Eq(sigma.1)
\FF_n\equiv \s\left(t_0, e_k(u), k\leq n, u\in \del \mathbf{T}\right).
\ee
We need to define further $\s$-algebras that correspond to events that take place in sub-trees. 
For a given multi-index $u$, define the set of multi-indices that coincide with $u$ in the first $n$ 
entries, 
\be\Eq(forward.1)
\UU_n(u)\equiv \left\{v\in \mathbf{T}: \forall k\leq n, v_k=u_k\right\}.
\ee
Naturally, this is the subtree that branches off the branch $u$ in the $n$-th generation.
Next, we define the $\s$-algebra generated by the times in these subtrees,
\be\Eq(forward.2)
\GG_n(u)\equiv \s\left(e_v(k), v\in \UU_n(u), k\leq n\right).
\ee 

\begin{definition}\TH(normal.1)
A set  $\EE_n(u)\in \GG_n(u)\subset$ is called \emph{normal}, if it is of the form 
\be\Eq(quasi.1)
\EE_n(u)=\{e_n(u)\leq  r\}\cap \EE_{n+1}((u(n),0))\cap \EE_{n+1}((u(n),1)),
\ee
or if
\be
\EE_n(u)= \R_+^{\mathbf{T}},
\ee
where the two events $\EE_{n+1}$ are normal.
We say that a normal event $\EE_n(u)$ has finite horizon, if there exists a $\infty>N\geq n$ such that 
$\EE_n(u)\in \FF_N$.
\end{definition}

\begin{definition}\TH(qm.1)
We say that the assignment of branching times is quasi-Markov (with time horizon $t$), if there is a family of probability measures, $Q_{t,T}, t\in\R_+, T\leq t$ on $\GG_0(0)$ and a family of probability measures $q_{t,T}, t\in \R_+, T\leq t$ on $(\R_+,\BB(\R_+))$ 
that have the following property.  For any event $\EE_0(0)\in \GG_0(0)$ which is of the form 
\be\Eq(quasi.1)
\EE_0(0)=\{e_0(0)\leq  r\}\cap \EE_{1}((0,0))\cap \EE_{1}((0,1)),
\ee
where $\EE_1(u)\in\GG_1(u)$, for all $t_0<r<t-t_0$,
\be\Eq(quasi.2)
Q_{t,t_0}(\EE_0(0))=\int_{t_0}^rq_{t,t_0}(ds) Q_{t,t_0+s}(\EE_1(\overleftarrow{00}))Q_{t,t_0+s}(\EE_1(\overleftarrow{01})).
\ee
\end{definition}

\begin{lemma} \TH(qm.2)
The measures $Q_{t,t_0}$ on the $\s$-algebra generated by the normal events with finite horizon in $\GG_0(u)$ are uniquely determined by the 
family of measures $q_{t,s}, s\leq t$. 
$q_{t,s}$ is the law of $e_u(n)$ conditioned on $T_{n-1}(u)=s$. 
\end{lemma}

\begin{proof}
From \eqv(quasi.2) it follows by simple iteration that the measure of any normal event of finite horizon is expressed 
uniquely in term of $q$. Noting further that the set of finite horizon events is intersection stable, the assertion follows from 
Dynkin's theorem.
\end{proof}
%
%

The total tree at continuous time $t$ is then described as follows:

\begin{itemize}
\item[(i)] The branches of the tree alive are
\be\Eq(alive.1)
\AA(t)\equiv\left\{ u(n): u \in \mathbf{T}, n\in \N_0 \;\text{s.t.}\;  T_{n-1}(u)\leq t<T_n(u)\right\}.
\ee  
\item[(ii)] The entire tree up        to time $t$ is the set 
\be\Eq(alive.2)
\TT(t)\equiv \left\{ u(k): k\leq n, u(n)\in \AA(t)\right\}.
\ee
\end{itemize}
Note that both sets are empty if $t<t_0$. It is a bit cumbersome to write, but the distribution of the set 
$\TT(t)$ together with the lengths of all branches can be written down explicitly in terms of the 
 laws $q$ and the branching laws of the underlying discrete-time tree.

\section{The simplified model as quasi-Markov Galton-Watson tree}

We return to the approximate model defined in Section 1. For simplicity, we keep the assumption that 
the underlying tree is binary. We first show that the branching times under the law $\wh P_{\l,t}$ 
define a quasi-Markov Galton-Watson tree. 

\begin{lemma} \TH(qm.10) The branching times of the simplified model under the law 
 $\wh P_{\l,t}$ are quasi-Markov, where $Q_{t,T}$ is the marginal distribution of 
 $\wh P_{\l,t-T}$
   with $q_{t,T} $ that is absolutely continuous
 w.r.t. Lebesgue mesure with density 
 \be\Eq(qm.11)
 \frac{\eee^{-s} \s(\l,\e) v_\s(t-s-T)^2}{v_\s(t-T)}\1_{s\geq T},
 \ee
 namely, 
 \be\Eq(qm.12)
\wh P_{\l,t-t_0}(\EE_0(0))=\int_{t_0}^rq_{t,t_0}(ds) \wh P_{\l,t-s}(\EE_1(\overleftarrow{00}))
\wh P_{\l.t-s}(\EE_1(\overleftarrow{01})).
\ee

 \end{lemma}
 
 \begin{proof} Let $\EE_0(0)=\{e_0(0)\leq r_0\}\cap \EE_1(00)\cap \EE_1(01)$.
      We now have
      \be\Eq(part.1)
      n(t)= \tilde n^{(00)} (t-T_1(00))+
       \tilde n^{(01)} (t-T_1(00)).
        \ee
        where  the $\tilde n  $ are the particle numbers in the four respective branches of the tree. 
        In analogy to \eqv(gen.2), we obtain
        \bea
       && v_\s(t-T)\wh P_{\l,t-T}(\EE_0(0))\nonumber\\
       &&=\int_T^t\eee^{-s+T}  \E \left[\s(\l,e)^{n^{(00)}(t-s)+n^{(01)}(t-s)-1}
        \EE_1(00)\EE_1(01)\right]ds\nonumber\\
        &&=\int_T^t\eee^{-s+T} \s(\l,\e) \E \left[\s(\l,e)^{n^{(00)}(t-s)-1}       \EE_1(00)\right]
        \E\left[\s(\l,\e)^{n^{(01)}(t-s)-1}
 \EE_1(01)\right]ds\nonumber\\
    &&=\int_T^t\eee^{-s+T} \s(\l,\e) v_\s(t-s)^2\frac{\E \left[\s(\l,e)^{n^{(00)}(t-s)-1}       \EE_1(00)\right]}{v_\s(t-s)}\frac{
        \E\left[\s(\l,\e)^{n^{(01)}(t-s)-1}
 \EE_1(01)\right]}{v_\s(t-s)}ds\nonumber\\
    &&=\int_T^t\eee^{-s+T} \s(\l,\e) v_\s(t-s)^2
    \wh P_{\l,t-s}\left( \EE_1(\overleftarrow{00})\right)
       \wh P_{\l,t-s}\left( \EE_1(\overleftarrow{01})\right)ds,
 \eea
where we used  the independence the events in the two branches under the 
original BBM measure $\EE$ and the definition of $\wh P_{\l,t}$.
This concludes the simple proof.
        \end{proof}

  \section{The limit $\l(t)\downarrow 0$}
  We have seen that a penalty with fixed $\l<\infty$ and $\e>0$ enforces that only a finite number 
  of branchings take place, even if we let $t$ tend to infinity. To get more interesting results, we consider now the case when 
  $\l=\l(t) $ depends on $t$ such that $\l(t)\downarrow 0$ as $t\uparrow \infty$. In fact, we will see that a rather interesting limiting model arises in this setting. 
  Clearly, in this case
 $\s(\l(t),\e) \approx \eee^{-\l(t)\e^2}\approx 1-\l(t)\e^2$ is  a good approximation.

We first look at the partition function.
\begin{lemma}\TH(zero.1)
Assume that $\l(t)\downarrow 0$. Then 
\be\Eq(zero.0)
\lim_{t\uparrow\infty} \eee^t \l(t) \e^2   v_{\s(\l(t),\e)}(t) =1.
\ee
\end{lemma}
\begin{proof}
We just use the explicit form of $v_\s(t)$ given in \eqv(pf.10). This gives
\bea\Eq(zero.3)
\eee^t \l(t) \e^2   v_{\l(t)}(t)&=&\frac{\l(t) \e^2}{1-\s(\l,\e) +\s(\l,\e)\eee^{-t}}\nonumber\\
&=&\frac{\l(t) \e^2}{\l(t)\e^2+O(\l(t)^2)) +O(\eee^{-t})}
=1+O(\l(t)),
\eea
which implies the statement of the lemma.
\end{proof}

From Theorem \thv(geo.1) we derive the asymptotics of the particle number.

\begin{theorem}\TH(zero.4)
Assume that $\l(t)\downarrow 0$, but $t+\ln (\l(t)\e^2)\uparrow \infty$, as $t\uparrow\infty$. Then: 
\begin{itemize} 
\item [(i)] The number of particles at time $t$ times $\l(t)\e^2$, $\l(t)\e^2n(t)$, converges in distribution to an exponential random variable with parameter $1$.
\item [(ii)] For any $\rho\in\R$, the number of particles at time $s(t)=t+\ln(\l(t)\e^2) +\rho$ converges in distribution to a geometric distribution with parameter $1/(1+\eee^\rho)$.
\item [(iii)] If $\rho(t)\uparrow\infty$ but $\ln(\l(t)\e^2) +\rho(t)\leq 0$ , the number of particles 
at time $s(t)=t+\ln(\l(t)\e^2) +\rho(t)$ divided by  $1+\eee^{\rho(t)}$ converges in distribution to an 
exponential random variable with parameter $1$.
\end{itemize}
\end{theorem}

\begin{proof}
The proof follows easily from the explicit computations of the Laplace tranforms of the particle numbers, see Eq. \eqv(s.1) and \eqv(s.5). 
\end{proof}

The next theorem gives the asymtotics of the first branching time.

\begin{theorem}\TH(zero.5) Let $\l(t)$ be as in Theorem \thv(zero.4). Then, for any $\rho\in \R$, 
\be\Eq(gen.8)
\lim_{t\uparrow\infty} \wh P_{\l(t),t} \left(\t_1\leq t+\ln(\l(t)\e^2)+\rho\right) 
=
 \frac 1{\eee^{-\rho}+1}.
 \ee
 \end{theorem}
\begin{proof}
From the explicit formula \eqv(branch.2),  we get that, 
 \be\Eq(gen.9)
 \wh P_{\l(t),t} \left(\t_1\leq t-r\right)=
 \frac{(1-\l\e^2)(\eee^{-r}-\eee^{-t})}{\l\e^2+\eee^{-r}(1-\l\e^2)}(1+O(\l^2\e^4))
 = \frac {1-\eee^{-(t-r)}}{\eee^r\l(t)\e^2+1}(1+O(\l^2\e^4)).
 \ee
 To get something non-trivial, the first term in the denominator should be of order one.
 That suggest to choose $r=r(t)=-\ln(\l(t)\e^2) -\rho$. 
 Eq. \eqv(gen.8) then follows directly.
 \end{proof}
 
 \subsection {The limiting Quasi-Markov Galton-Watson tree}
 Theorem \thv(zero.5) 
  also suggests to define 
 \be\Eq(yes.1)
 \tilde \t_1\equiv \t_1-t-\ln (\l(t)\e^2).
 \ee
 $\tilde\t_1$ should  be thought of as the position of the first branching seen from the standard position 
 $t +\ln(\l(t)\e^2)$. 

 To derive the asymptotics of  the consecutive branching times, we just have to 
 look at 
  \be\Eq(two.1)
  \wh P_{\l(t),t} (e_u(n+1)\leq \D|\FF_n)=
  \wh P_{\l(t), t-T_u(n)}(\t_1\leq \D).
    \ee
    For this we have from the previous computation
    \be\Eq(two.1.1)
    \wh P_{\l(t), t-T_u(n)} \left( \t_1\leq \D\right) =\frac{1-\eee^{-\D}}{\eee^{t-T_u(n)+\ln(\l(\e^2)} \eee^{-\D} +1}.
    \ee
    Recall that, e.g. $T_u(1)=e_u(1)= t+\ln(\l\e^2)+\rho$,  where $\rho$ is a finite random variable,
    so that we will have that,  in general, $ T_u(n)=e_u(1)= t+\ln(\l\e^2)+O(1)$, so that the right-hand side of 
    \eqv(two.1.1) will converge to some non-trivial distribution.

The asymptotic results above suggest to consider the branching times of the process in the limit 
$t\uparrow \infty$, $\l(t)\downarrow 0$, around the  time $t-\ln (\l(t)\e^2)$. 
We have seen that the time of the first branching shifted by this value converges in distribution 
 to a random variable with distribution 
function $1/\left(\eee^{-\rho}+1\right)$ (which is supported on $(-\infty, \infty)$).

This suggests to define a limiting model as a quasi-Markov Galton-Watson tree with  the
measures 
\be\Eq(qml.1)
q_{\infty,T}(e\leq \D)=\frac {1-\eee^{-\D}}{\eee^{-\D-T}+1}.
\ee
This gives, in particular, for the first branching time,
\be
Q_{\infty, t_0}(e_0(0)\leq \D)=\frac {1-\eee^{-\D}}{\eee^{-\D-t_0}+1}.
\ee
We have to choose $t_0$ to match this with the known asymptotics of 
the first branching time, see \eqv(gen.8). It turns out that
\be\Eq(qml.3)
\lim_{t_0\downarrow -\infty}  Q_{\infty, t_0}(e_0(0)\leq -t_0+\D)
=\lim_{t_0\downarrow -\infty} 
\frac {1-\eee^{t_0 -\D}}{ \eee^{-\D}+1}=\frac 1{1+\eee^{-\D}},
\ee
for all $\D\in \R$. 
So the picture is that we start the process at $t_0=-\infty$ and the first branching time is infinitely far in the future and occurs at a finite random time distributed according to \eqv(qml.3). 
The density of this distribution is $\frac 14 \cosh(\D)^{-2}$. In particular, it has mean zero and variance $\pi^2/3$.

We have the following result.
\begin{theorem}\TH(limit-th.1)
Assume that $\l(t)\downarrow 0$ and  $t+\ln(\l(t)\e^2)\uparrow\infty$, as $T\uparrow\infty$. Then, for any $\D\in \R$ and events $\EE_1(u)\in \GG_1$,
\bea\Eq(limit-th.2)
&&\lim_{t\uparrow\infty}
\wh P_{\l(t),t} \left(\{e_0(0)\leq \D+t+\ln(\l(t)\e^2)\}\cap \EE_1((0,0))\cap \EE_1((0,1))
\right)\nonumber\\
&&= Q_{\infty,-\infty} \left(\{e_0(0)\leq \D\}\cap \EE_1((0,0))\cap \EE_1((0,1))\right),
\eea
where $Q_{\infty,-\infty}$ is the law of the limiting  model.
\end{theorem}

\begin{proof} All we need to show is that the measures $q$ converge. But this follows form the computations indicated above.
\end{proof}

\begin{figure} \label{fig.2}
\begin{center}
\includegraphics[width=8cm]{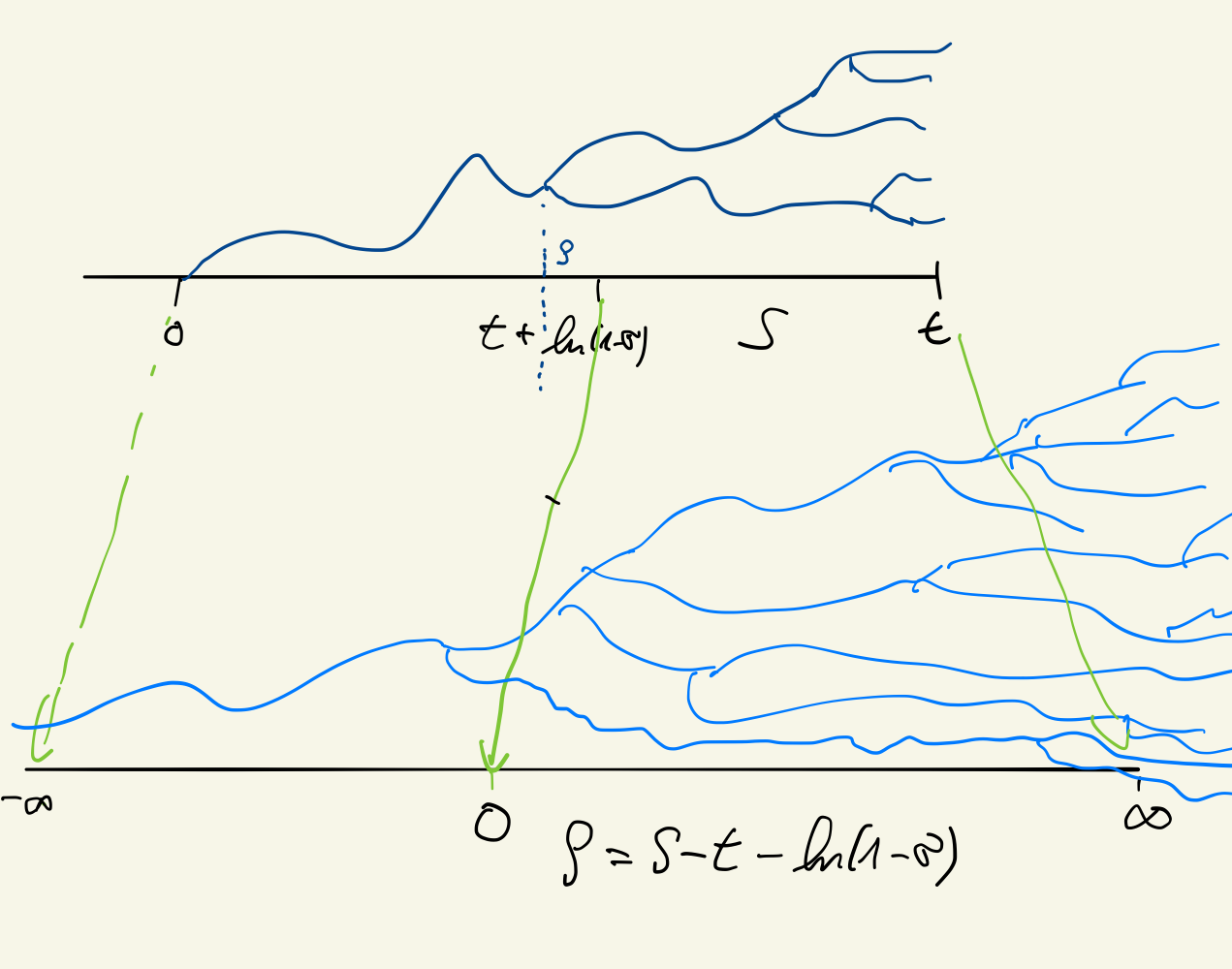}
\caption{Scaling towards the limit process}
\end{center}
\end{figure}
Note further, as long as $T_u(n)$ is negative, the distribution of $e_u(n)$ is concentrated around
$T_u(n)$, while as $T_u(n)$ become positive and large, the distribution tends to a standard exponential distribution.
In fact, 
\be\Eq(law.4)
\E\left[e_u(n+1)|\FF_n\right] = \left(1+\eee^{+T_u(n)}\right)\ln\left(1+\eee^{-T_u(n)}\right).
\ee
Clearly this converges to $1$, as $T_u(n)\uparrow\infty$ and behaves like $-T_u(n)$, as $T_u(n)$ tends to $-\infty$.

%
%
 

\section{The distribution of the front}

An obvious first question is the distribution of the maximum of BBM under the law $\wh P_{\l,t}$.
We define, for any $z\in \R$, 
\be
\Eq(max.1)
u_\s (t,z)=\E\left[\s(\l,\e)^{n(t)-1}\1_{\forall_{1=1}^{n(t)} x_i(t)\leq z}\right].
\ee
Then 
\be\Eq(prob.1)
\wh P_{\l,t}\left(\forall_{1=1}^{n(t)} x_i(t)\leq z\right)=\frac{u_\s (t,x)}{v_\s(t)}\equiv 1-w_\s(t,x).
\ee
Note that we use the  choice $1-w_\s$ to be closer to the usual formulation of the F-KPP equation.

Interestingly, $w_\s$ solves a time-dependent version of the F-KPP equation.
\begin{lemma}\TH(FKPP.0)
$w_\s$ defined in \eqv(prob.1) is the unique solution of the equation
\be\Eq(better.1)
\del_t w_\s 
=\frac 12\del_{xx} w_\s  + w_\s (1-w_\s) \frac {\s}{(1-\s)\eee^{t}+\s}.
\ee
with initial condition $w_\s(0,x)=\1_{x\leq 0}$.
\end{lemma}

\begin{proof}
In complete analogy to the derivation of the F-KPP equation (see, e.g. \cite{bbm-book}), $u_\s $ satisfies the recursive equation
\be
\Eq(max.2)
u_\s (t,z)=\eee^{-t} \Phi_t(z)+\int_0^t ds \eee^{-(t-s)} \int dy\frac {\eee^{-\frac{y^2}{2(t-s|)}}}{\sqrt{2\pi(t-s)}}\s(\l,\e)
u_\s (s,z-y)^2,
\ee
where $\Phi_t(z)=\int_{-\infty}^z \frac {\eee^{-\frac {x^2}{2t}}}{\sqrt {2\pi t}}dx$ is the probability that 
a single Brownian motion at time $t$ is smaller than $z$.
Letting $H(t,x)\equiv\frac{ \eee^{-t -\frac{x^2}{2t}}}{\sqrt {2\pi t}}$,
we can write this as
\be
\Eq(max.3)
u_\s (t,z)= \int_{-\infty}^\infty dy H(t, z-y)\1_{y\leq 0} +\int_0^t ds \int dy H(s,y) \s(\l,\e)u_\s (t-s,z-y)^2.
\ee
Note that $H$ is the Green function for the differential operator $\del_t-\frac 12 \del_{xx}+1$.
and so $u_\s $ is the mild formulation of the partial differential equation 
\be\Eq(pde.1)
\del_t u_\s (t,z)=\frac 12 \del_{zz} u_\s (t,z)-u_\s (t,z)+\s(\l,\e) u_{\l}(t,z)^2,
\ee
with initial condition $u_\s (0,z)=\1_{z\geq 0}$.
This equation is the F-KPP equation if $\s(\l,\e)=1$, i.e. if $\l=0$, and looks similar to it in general.
 Hence, 
\bea\Eq(better.1.1)
\del_t w_\s &=&-\frac{\del_t u_\s }{v_\s}+\frac {u_\s \del_t v_\s}{v_\s^2}
=\frac 12\del_{xx} w_\s  +\s w_\s (1-w_\s )v_\s\nonumber\\
&=&\frac 12\del_{xx} w_\s  + w_\s (1-w_\s ) \frac {\s}{(1-\s)\eee^{t}+\s},
\eea
where we used the explicit form of $v_\s$ from \eqv (pf.10).
\end{proof} 

Note that \eqv(better.1) is a time-dependent version of the F-KPP equation, where the non-linear term is modulated down over time. 
Time dependent F-KPP equations have been studied in the past, see, e.g. \cite{Rossi2014,Ryzhik2013,Ryzhik2015}, but we did not find 
this specific example in the literature.
For small $\l$, \eqv(better.1) becomes
\be\Eq(better.2)
\del_t w_\s 
=\frac 12\del_{xx} w_\s  + w_\s (1-w_\s )\frac {1+O(\l\e^2)} {\l\e^2\eee^t+1}.
\ee

For future use, note that \eqv(better.2) is a special case of a class of F-KPP equations of the form 
\be\Eq(general.1)
\del_t \psi=\frac 12\del_{xx} \psi + g(t)\psi(1-\psi),
\ee
where $g:\R\rightarrow \R_+$. We will be interested in the case when $g$ is bounded, monotone decreasing, and integrable.
%
%
%

The key tool for analysing solutions of \eqv(better.1) is  the Feynman-Kac representation for $\psi$, see Bramson \cite{B_C}.
\begin{lemma} \TH(feyn.1)
If $\psi$ is a solution of the equation \eqv(general.1) with initial condition $\psi(0,x)=\rho(x)$,
then $\psi$ satisfies 
\be\Eq(feyn.2)
\psi(t,x)=\E_x\left[\exp\left(\int_0^t g(t-s)(1-\psi(t-s,B_s))ds\right) \rho(B_t)\right],
\ee
where $B$ is Brownian motion starting in $x$.
\end{lemma}

The  strategy to exploit this representation used by Bramson is to use a priori bounds on $\psi$ in the right-hand side of the equation
in order to get sharp upper and lower bounds. Here we want to do the same, but we need to take into account the specifics
 of the function $g$. Going back to the specific case \eqv(better.2), $g$ remains close to $1$ for a fairly long time ($O(-\ln (\l\e^2))$), 
 and then decays exponentially with rate $1$ to zero. Therefore, we expect that initially, the solution will behave like that of the F-KPP equation and approach a travelling wave solution. As time goes on, the wave slows down and comes essentially to a halt. Finally, as time goes on we see a pure diffusion. 
 We will deal differently with these three regimes. 
 We begin with the initial phase when $g(t)\sim 1$.

\begin{lemma} \TH(constant.1)
Assume that $g$ is non-increasing and bounded by one and zero from above and below. 
Define $G(t)=\int_0^tg(s) ds$. 
Then 
\be
\eee^{G(t)-t} w_0(t,x)\leq w_\s(t,x) \leq w_0(t,x), \quad\forall x\in\R, t\in\R_+.
\ee
\end{lemma}
\begin{proof}
Starting from \eqv(feyn.2),
we see that 
\bea\Eq(feyn.2-1)
\psi(t,x)&=&\E_x\left[\exp\left(\int_0^t
\left (1-\psi(t-s,B_s))ds\right) +
 (g(t-s)-1)(1-\psi(t-s,B_s))ds\right) \rho(B_t)\right]\nonumber\\
 &\geq& \exp\left(\int_0^t (g(t-s)-1)ds\right)
 \E_x\left[\exp\left(\int_0^t
\left (1-\psi(t-s,B_s))ds\right) ds\right) \rho(B_t)\right]\nonumber\\
&=&\exp(G(t)-t) \psi_0(t,x).
\eea
The upper bound follows since $g(t-s)\leq 1$ in the same way.
 \end{proof}
 
 $G$ can be computed explicitly for $g(t)=\s/((1-\s) \eee^t+\s)$, 
 namely 
 \be\Eq(geee.1)
 G(t)= t-\ln \left(1+(1/\s-1)\eee^t\right)+\ln \left(1/\s)\right)= t-\ln\left(\s +(1-\s) \eee^{t}\right). 
 \ee
 Notice that
 \be
 \lim_{t\uparrow \infty} G(t)=-\ln (1-\s).
 \ee
 Define, for $\d>0$,
\be\Eq(constant.2)
\t_\d\equiv \sup\left\{t>0:  1-g(t)\leq \d\right\}.
\ee 
Obviously,
\be
\frac{\s}{(1-\s) \eee^{\t_\d}+\s} =1-\d,
\ee
 so 
 \be
 \t_\d=-\ln (1/\s-1)-\ln(1/\d-1).
 \ee 
 Finally, $G(\t_\d) =\t_\d+\ln(1-\d)+\ln (1/\s)$.
 
 In the limit $\l\downarrow 0$, we get 
 \be
 \t_\d\sim -\ln(\l\e^2)-\ln(1/\d-1),
 \ee
 \be\Eq(geee.1)
 G(\t_\d)-\t_\d\sim \ln (1-\d),
 \ee
 and 
 \be 
 \lim_{t\uparrow\infty}G(t)\sim -\ln (\l\e^2).
 \ee
 
 We see that, as $\l\downarrow 0$, $\t_\d\uparrow \infty$. This allows us to deduce the precise behaviour of the solution 
 at this time via Bramson's results.

\begin{lemma}\TH(wave.101)
If $w_\s$ satisfies \eqv(better.1) with Heaviside initial condition 
Then,  as $\l\downarrow 0$, 
\be\Eq(wave.102)
(1-\d) v_0(x)\leq 
w_\s \left(\t_\d, x+  m(\t_\d)\right)
\leq v_0(x),
\ee
where $v_0$ is a travelling wave of the F-KPP equation with speed $\sqrt{2}$
\be\Eq(wave.1)
\frac{1}{2}\del_{xx} v_0 +\sqrt{2} \del_x v_0 +v_0(1-v_0)=0,
\ee
and  $m(t)\equiv \sqrt 2 t-\frac 3{2\sqrt 2} \ln t$.
\end{lemma}
\begin{proof} This follows immediately from Bramson's theorems A and B in \cite{B_C}, Lemma \thv(constant.1), and 
\eqv(geee.1). 
\end{proof}

Next we look at the behaviour of the solution for times when $g(t)\ll1$. 

 \begin{lemma} \TH(diffuse.1)
 Let $\psi$ solve \eqv(general.1) and $g$ be integrable. Define, for $\D>0$, $T_\D$ by
 \be\Eq(eq.T)
 T_\D=\inf\left\{t>0: \int_t^\infty g(s)ds\leq\D \right\}.
 \ee
 Then, for  $t>T_\D$,
 \be\Eq(plopp.2)
 \E_x \left[ w_\s(T_\D,B_{t-T_\D})\right]\leq
 w_\s(t,x)\leq \eee^\D \E_x \left[ w_\s(T_\D,B_{t-T_\D})\right],
 \ee
 where $B$ is Brownian motion started in $x$.
 \end{lemma}
 
 \begin{proof} We have that, for $t\geq T_\D$,
\be\Eq(plopp.1)
 w_\s(t,x) =\E_x\left[\exp\left(\int_0^{t-T_\D} g(t-s) (1-w_\s(t-s,B_s))ds\right)w_\s(T_\D,B_{t-T_\D})\right].
 \ee
 The exponent is trivially bounded by
 \be
 0\leq  \int_0^{t-T_\D} g(t-s)(1-w_\s(t-s,B_s))ds\leq  \int_0^{t-T_\D} g(t-s)ds
 = \int_{T_\D}^t  g(s)ds \leq \D
 \ee
Inserting these bounds into \eqv(plopp.1) gives \eqv(plopp.2). 
\end{proof}

Note that in our case,  $T_\D$ is determined by 
\be
G(\infty)-G(T_\D)= \D.
\ee
But 
\bea
G(\infty)-G(T_\D)&=& -\ln(1-\s) -T_\D+\ln \left(\s+(1-\s)\eee^{T_\D}\right)\\\nonumber
&=&-\ln(1-\s) +\ln \left(\s\eee^{-T_\D}+(1-\s)\right)
= \ln \left(\s\eee^{-T_\D}/(1-\s)+1\right)
\eea
Now set $T_\D=-\ln (1-\s)+z$. Then
\be
G(\infty)-G(T_\D)=\ln \left(\s\eee^{-z}+1\right).
\ee
Hence 
\be\Eq(teee.1)
T_\D=-\ln(1/\s-1) -\ln \left(\eee^\D-1\right)\sim \ln(1/\s-1) +\ln(1/\D),
\ee
for small  $\D$. In particular, we have that 
\be
T_\D-\t_\d\sim  \ln \left(\eee^\D-1\right)+\ln (1/\d-1) \sim  \ln(1/\D)+\ln (1/\d),
\ee
for small $\D$ and $\d$. 

What is left to do is to control the evolution of the solution between time $\t_\d$ and $T_\D$.
The Feynman-Kac representation, for $\t_\d\leq t\leq T_\D$,
\be
\Eq(in-between.1)
w_\s(t,x) =\E_x\left[ \exp\left(- \int_0^{t-\t_\d} g(t-\t_\d-s)(1-w_\s(t-\t_\d-s,
B_s))\right)
w_\s(\t_\d,B_{t-\t_\d})\right].
\ee
To start, the following bounds are  straightforward from
  the Feynman-Kac representation and the fact that $w_\s\in[0,1]$.
  
  \begin{lemma} \TH(upper-triv.1) With the notation above, for $\t_\d\leq t\leq T_\D$, 
  \be\Eq(upper-triv.2)
   \E_x\left[w_\s(\t_\d,B_{t-\t_\d})\right]\leq 
w_\s(t,x) \leq \eee^{G(t)-G(\t_\d)}    \E_x\left[w_\s(\t_\d,B_{t-\t_\d})\right]\land 1,
\ee
where 
\be
G(t)-G(\t_\d)=-\ln\left((1-\d)\eee^{-(t-\t_\d)}+\d\right).
\ee
and $\eee^{G(T_\D)-G(\t_\d)} \sim\frac 1{\d}\eee^{-\D}$.
\end{lemma}

\begin{remark} 
We expect the upper bound to be closer to the correct answer. 
\end{remark}
We now combine all estimates. This gives, for $t\geq T_\D$, 
\be\Eq(plopp.3)
(1-\d)
\E_0\left[v_0(x+B_{t-\t_\d})\right]
\leq w_\s(x+m(\t_\d))\leq (1/\d) \E_0\left[v_0(x+B_{t-\t_\d})\right]\land 1.
\ee
If we choose, e.g. $\d=1/2$, we see that the upper and lower bounds only differ by a factor $4$.

The expectation over $v_0$ can be bounded using the known tail estimates (see \cite{B_C} and 
\cite{CHL17}), 
\bea\Eq(chl.1)
v_0(x)&\leq& Cx\eee^{-\sqrt 2 x},  \;\hbox{if}\, x>1,\\
v_0(x) &\geq& 1-C\eee^{(2-\sqrt 2)x},  \;\hbox{if}\, x<-1,
\eea
However, the resulting expressions are not very nice and not very precise, so we leave their computation to the interested reader.

We conclude this chapter by summarising the behaviour of the solution as a function of $t$ when $\l\downarrow 0$.

\begin{theorem}\TH(thm.equation)
Assume that $\l\downarrow 0$.  Let $0<\d<1$ and $\t_\d$ defined as in \eqv(constant.2). Then 
\be
\t_\d = \ln (1/\s-1)-\ln(1/\d-1)\sim -\ln(\l\e^2)-\ln(1/\d-1).
\ee
Moreover, for $\D>0$, let  $T_\D$ be defined  in \eqv(eq.T). Then,
\be
T_\D=-\ln(1/\s-1) -\ln \left(\eee^\D-1\right)\sim -\ln(\l\e^2) +\ln(1/\D),
\ee
for $\l\downarrow 0$ and $\D$ small.
Then the solution $w_\sigma$ of \eqv(better.1) can be described as follows.
\begin{itemize}
\item[(i)] For $0\ll t\leq \t_\d$ 
\be\Eq(eq.full.1)
w_\s(t,x+m(t)) \sim v_0(x),
\ee
where $v_0$ is the solution of \eqv(wave.1).
\item[(ii)] For  $\t_\d<t<T_\D$ we have
\be\Eq(eq.full.2)
(1-\d) \E_x\left[v_0(B_{t-\t_\d})\right]\leq 
w_\s(t,x+m(\t_\d)) \leq \eee^{G(t)-G(\t_\d)}    \E_x\left[v_0(B_{t-\t_\d})\right]\land 1,
\ee
\item[(iii)] For $t \geq T_\D$, we have
\be
(1-\d) \E_x\left[v_0(B_{t-\t_\d})\right]\leq 
w_\s(t,x+m(\t_\d)) \leq \frac{1}{\d}  \E_x\left[v_0(B_{t-\t_\d})\right]\land 1.
\ee
\end{itemize}
\end{theorem}
This picture corresponds to the geometric picture we have established in the preceding sections, in a sort of time reversed way: the diffusive behaviour at large times corresponds to the Brownian morion up to the time of the first 
branching $(\sim + \ln(\l\e^2)$, the travelling wave behaviour at times up to $\t_\d$ corresponds to the almost freely 
branching at the late times after $t +\ln (\l\e^2)$, and the finite time interval between $\t_\d$ and $T_\D$, when the 
travelling waves comes to a halt corresponds to the  first branching steps that are asymptotically described by the
limiting quasi-Markov Galton-Watson tree described in Section 5.

\section{Comparison to the full model}

We will show that in the original model, with the interaction given by $I_t$ (see Eq. \eqv(int.1)),
behaves similarly to the simplified model. In particular, the first branching happens at least as late as in that model. 

\begin{lemma}\Eq(full.1)
 Let $\t_1$ be the first branching time. Then 
\be\Eq(real.10)
P_{t,\s}  \left(\t_1\leq t- r\right)\leq \frac{\s(\eee^{-r}-\eee^{-t})}{(1-\s (1-\eee^{-t}))(1-\s (1-\eee^{-r}))}.
 \ee
 For $\l\downarrow 0$ and $t\uparrow \infty$, this behaves as
  \be\Eq(real.10.2)
P_{t,\s}  \left(\t_1\leq t- r\right)\leq \frac{\eee^{-r}}{\l\e^2 (\l\e^2+\eee^{-r})}=\frac 1{\l^2\e^4\eee^{r} +\l\e^2}.
 \ee
and so
\be\Eq(real.11)
 P_{t,\s}  \left(\t_1\leq t +2\ln(\l(t)\e^2) -\rho\right)\leq  \eee^{-\rho},
 \ee
\end{lemma}

\begin{proof}
Set
\be\Eq(real.1)
V(t,r)\equiv \E \left[\eee^{-\l I_t(x)}\1_{\t_1\leq r}\right].
\ee
Then
\be\Eq(real.2)
P_{t,\s}  \left(\t_1\leq t- r\right)= \frac{V(t,t-r)}{\E \left[\eee^{-\l I_t(x)}\right]}
\leq\eee^{t}v(t,t-r),
\ee
where we simply bounded the denominator by the probability that there is no branching up to time 
$t$.
Inserting the explicit form of $v_\s(t,t-r)$ gives \eqv(real.10). The asymptotic formulae for small $\l$ are straightforward.
\end{proof}
%

%

 One can improve the bound above as follows.  Instead of bounding the denominator just by the probability that there is no branching up to time $t$, we can bound it by no branching up to time $t-q$
 and then bound the interaction of the remaining piece uniformly  by
 \be\Eq(real.5)
I_q(x)\leq \int_0^q n(s)(n(s)-1)ds.
\ee
Hence the denominator becomes 
\be\Eq(real.3)
\E \left[\eee^{-\l I_t(x)}\right]\geq 
\E \left[\eee^{-\l I_t(x)}\1_{\t_1>t-q}\right] = \eee^{-t+q}
\E \left[\eee^{-\l I_q(x)}\right].
\ee
Now
\bea\Eq(real.6)
\E \left[\eee^{-\l I_q(x)}\right]&\geq& \E \left[\eee^{-\l \int_0^q n(s)(n(s)-1)ds} \1_{n(s)\leq c\E n(s) \forall_{s\leq q}}\right]
\nonumber\\&\geq&
\eee^{-\l \int_0^q c\E n(s)( c\E n(s)-1)ds} \E \left[ \1_{n(s)\leq c\E n(s) \forall_{s\leq q}}\right].
 \eea
 Since $n(s)/\E n(s)$ is a positive martingale, by Doob's maximum inequality we have that, for $c>1$,
 \be\Eq(real.7)
 \E \left[ \1_{n(s)> c\E n(s) \forall_{s\leq q}}\right]
 \leq
 1/c.
 \ee
 Moreover, $\E[n(s)]=\eee^s$, and hence
 \be\Eq(real.8)
 \E \left[\eee^{-\l I_q(x)}\right]\geq\eee^{-t+q} (1-1/c) \eee^{-\l c^2 \eee^{2q}/2}.
 \ee
 Finally, we make the close to optimal choice    $q= \frac12 \ln (1/c^2\l)$, which 
 yields 
  \be\Eq(real.8.1)
 \E \left[\eee^{-\l I_q(x)}\right]\geq\eee^{-t+ \frac12 \ln (1/c^2\l)} (1-1/c) \eee^{-1} = \eee^{-t} (\l c^2)^{-1/2} (1-1/c)/\eee .
 \ee
 Thus, choosing $c=2$,
 \eqv(real.11) improves to 
  \be\Eq(real.10.2)
P_{t,\s}  \left(\t_1\leq t- r\right)\leq 4\eee \frac {\sqrt {\l}}{\l^2\e^4\eee^{r}+\l\e^2 \;} .
 \ee
 Hence,
 \be
 P_{t,\s}  \left(\t_1\leq t+\ln(\l^{3/2} \e^4) -\rho
 \right)\leq \frac {4\eee }{\eee^{\rho}+\sqrt{\l}\e^2}\sim 4\eee^{-\rho+1}.
 \ee
This is still not perfect for small $\l$, but it seems very hard to improve the bound on the denominator 
much more. Improvement would need to come from a matching  bound in the numerator.

\section{The case $p_0>0$}

The behaviour of the model is very different if particles are allowed to die. In that case, 
the process will die out almost surely. However, it is still interesting to see how exactly this happens.
To simplify things, we assume in the sequel $p_0>0$ and $p_2=1-p_0$. 
Note first that the approximate  penalty function changes slightly, since now the number of branching events is no longer related to the number of particles. Let us introduce the two numbers
$m(t)$ and $d(t)$ as the number of  (binary) 
branchings and deaths, resp. that occurred up to time $t$. Clearly, $n(t)=1+m(t)-d(t)$. We then have
\be\Eq(approx.1)
I_t(x) \geq \sum_{i=1}^{m(t)} \t_{\e}(i),
\ee

We consider first the partition function
function $v_\s(t)=\E \left[\s(\l,\e)^{m(t)}\right]$. The analog of Lemma \thv(pf.2)
is as follows.

\begin{lemma}\TH(pf0.2) Let
$\tilde v_\s(t)$ be the solution of the ordinary differential equation 
\be\Eq(pf0.3)
\frac d{dt} \tilde v_\s(t) =  \s(\l,\e) p_2  \tilde v_\s(t)^2 -\tilde v_\s(t) +p_0, 
\ee
with initial condition $\tilde v_\s(0)=1$.
Then 
$ v_\s(t)=\tilde v_{\s(\l,\e)}(t)$.
\end{lemma}

\begin {proof} We proceed as in the proof of Lemma \thv(pf.2). Since now the first event could be either a branching (with probability $p_2$) or a death (with probability $p_0$), 
we get the recursion
\be\Eq(pf0.4)
v_\s(t) =\eee^{-t} + \int_0^t ds \eee^{-(t-s)} \left(p_2\s(\l,\e) v_\s(s)^2 +p_0\right).
\ee
Differentiating yields the asserted claim.
\end{proof}

The presence of the term $p_0>0$ eliminates the fixpoint $0$ in equation \eqv(pf0.3). 
In fact, \eqv(pf0.3) has the two fixpoints
\be\Eq(pf0.5)
v^\pm_\s \equiv \frac 1{2sp_2}\left(1\pm \sqrt{1-4sp_2+4sp_2^2}\right)
\ee

Note that for $\s(\l,\e)=1$, this simplifies to 
\be\Eq(s=1.1)
v^\pm_0=\frac  1{2p_2}\left(1\pm \sqrt{(1-2p_2)^2}\right),
\ee
which is $1$ and $p_0/p_2$. In that case we clearly have $v_0(t)=1$ in all cases.

If $\l>0$, but $\l\ll 1$ (i.e. $\s(\l,\e)<1$, but $1-\s(\l,\e)$ small), we can expand 
\be\Eq(fp0.6)
v_\s^\pm =
\begin{cases}
 \frac 1{2\s(\l,\e)p_2}\left(1\pm(2\s(\l,\e)p_2-1) \sqrt{1+\frac{4p_2^2\s(\l,\e)(1-\s(\l,\e))}{(2p_2\s(\l,\e)-1)^2}}\right), &\text{if} \;\; p_2> 1/2,\\
   \frac 1{2\s(\l,\e)p_2}\left(1\pm(1-2\s(\l,\e)p_2) \sqrt{1+\frac{4p_2^2\s(\l,\e)(1-\s(\l,\e))}{(1-2p_2\s(\l,\e))^2}}\right), &\text{if} \;\; 
   p_2\leq 1/2.
   \end{cases}
   \ee
   In particular, the smaller fixpoint is 
\be\Eq(fp0.7)
v_\s^- \approx
\begin{cases}
 \frac {p_0}{p_2}  +O((1-\s(\l,\e))^2), &\text{if} \;\; p_2> 1/2,\\
   1 -\frac{p_2(1-\s(\l,\e))}{1-2p_2\s(\l,\e)}, &\text{if} \;\; 
   p_2\leq 1/2.
   \end{cases}
   \ee
Now set $f_\s(t)\equiv v_\s(t)-v^-_\s$. Then $f_\s$ satisfies the differential equation
\be\Eq(dg0.1)
\del_t f_\s(t)=\s(\l,\e)p_2f_\s(t)^2+(2\s(\l,\e)p_2v^-_\s-1)f_\s(t),
\ee
with initial condition $f_\s(0)=1-v^-_\s$. 
We can solve this equation as in the case $p_0=0$. Define
\be
\hat f_\s(t)\equiv \eee^{-(2\s(\l,\e)p_2v^-_\s-1)t}f_\s(t).
\ee
Then 
\be\Eq(dg0.2)
\del_t \hat f_\s(t)=\s(\l,\e)p_2\hat f_\s(t)^2\eee^{(2\s(\l,\e)p_2v^-_\s-1)t},
\ee
which has the solution 
\be\Eq(solu.1)
\hat f_\s(t)=\frac 1{\frac 1{1-v^-_\s}-\frac {\s(\l,\e)p_2}{1-2\s(\l,\e)p_2v^-_\s}\left(1-\eee^{(2\s(\l,\e)p_2v^-_\s-1)t}\right)}.
\ee
Hence 
\be\Eq(solu.1)
 f_\s(t)=\frac {\eee^{(2\s(\l,\e)p_2v^-_\s-1)t}}
 {\frac 1{1-v^-_\s}-\frac {\s(\l,\e)p_2}{1-2\s(\l,\e)p_2v^-_\s}\left(1-\eee^{(2\s(\l,\e)p_2v^-_\s-1)t}\right)}.
\ee
Note that in the case $\l=0$, this is just one, while otherwise, it decays exponentially to zero, so that 
$v_\s(t)\rightarrow v_\s^->0$, indicating that the number of branchings in the process remains finite.

Next let us consider the generating function of the particle number, 
\be
\Eq(ww.1)
w_{\s,\g}(t)\equiv \E \left[\s(\l,\e)^{m(t)} \eee^{-\g n(t)}\right].\
\ee
We readily see that this function satisfies the equation
\be\Eq(www.2)
w_{\s,\g}(t) = \eee^{-\g} \eee^{-t}+ p_2\s(\l,\e) \int_0^t \eee^{-(t-s)} w_{\s,\g}(t)^2 ds
+p_0 \left(1-\eee^{-t}\right).
\ee
This implies the differential equation
\be\Eq(ww.3)
\del_t w_{\s,\g}(t) =p_2\s(\l,\e)w_{\s,\g}(t)^2-w_{\s,\g}(t)+p_0,
\ee
with initial condition $w_{\s,\g}(0) =\eee^{-\g}$. 
Thus, $w$ and $v$ differ only in the initial conditions. 
it is therefore easy to see that 
\be\Eq(ww.4)
w_{\s,\g}(t) = v_\s^{-} +\frac {\eee^{(2\s(\l,\e)p_2v^-_\s-1)t}}
 {\frac 1{\eee^{-\g}-v^-_\s}-\frac {\s(\l,\e)p_2}{1-2\s(\l,\e)p_2v^-_\s}\left(1-\eee^{(2\s(\l,\e)p_2v^-_\s-1)t}\right)}.
\ee
From this expression we can compute, e.g., the expected number of particles at time $t$ under the 
measure $\hat P_{\s,t}$,
\be\Eq(ww.5)
\hat E_{\s,t} [n(t)] =- \frac {\del}{\del \g} \ln \left(w_{\s,\g}(t)\right)\big |_{\g=0},
\ee
which reads
\be\Eq(ww.6)
\hat E_{\s,t} [n(t)] =\frac{1}{v_\s(t)} 
\frac {\eee^{(2\s(\l,\e)p_2v^-_\s-1)t} }  {\left(
1-\frac {\s(\l,\e)p_2(1-v^-_\s)}{1-2\s(\l,\e)p_2v_\s^-}\left(1-\eee^{(2\s(\l,\e)p_2v^-_\s-1)t}\right)\right)^2}.
\ee
For $t\uparrow\infty$, this behaves to leading order, provided $v_\s^->0$, like
\be\Eq(ww.7)
\frac{1}{v_\s^-} \frac {\eee^{(2\s(\l,\e)p_2v^-_\s-1)t} }  {\left(
1-\frac {\s(\l,\e)p_2(1-v^-_\s)}{1-2\s(\l,\e)p_2v_\s^-}\right)^2}.
\ee
This implies that the process dies out exponentially fast unless the death rate is zero. This does not
come, of course, as a surprise.

\paragraph{Alternative computation of $\hat E_{\s,t}[n(t)]$.}

Instead of passing through the generating function for $n(t)$, we can also proceed by deriving a direct 
recursion for $\hat E_{\s,t}[ n(t)]$. To do so, define the un-normalised expectation 
\be\Eq(alt.1)
u_\s(t) =\E  \left[n(t)  \s(\l,\e)^{m(t)}\right].
\ee
Clearly we have
\be\Eq(alt.2)
u_\s(t)= \eee^{-t} + p_2 \int_{0}^t \eee^{-(t-s)} \s(\l,\e) 2u_\s(s) v_\s(s)ds,
\ee
where we used that in the case when the first event is a death at time $t-s$, 
$n(t)$ will be zero, while in the case
of a birth, $n(t) = (n_1(s)+n_2(s))$, where $n_i$ are independent copies.
This implies the differential equation 
\be\Eq(alt.3)
\del_t u_\s(t)= 2p_2\s(\l,\e) u_\s(t)v_\s(t)-u_\s(t).
\ee
The solution of this can be written directly as
\be
\Eq(alt.4)
u_\s(t) =\exp\left(\int_0^t \left(2p_2\s(\l,\e) v_\s(s)-1\right)ds\right).
\ee
Since $v_\s$ is explicit, one can verify that this gives the same answer as \eqv(ww.6).

\end{document}